\documentclass[12pt,a4paper,leqno]{amsart}

\input amssymb.sty

\newtheorem{thm}{Theorem}
\newtheorem{theorem}{Theorem}
\newtheorem*{theorem*}{Theorem}
\newtheorem{lemma}[thm]{Lemma}
\newtheorem*{thm*}{Theorem}
\newtheorem*{thmA*}{Theorem A}
\newtheorem*{thmB*}{Theorem B}
\newtheorem*{theoremL*}{Theorem L}
\newtheorem*{theoremPL1*}{Theorem PL1}
\newtheorem*{theoremPL2*}{Theorem PL2}
\newtheorem*{TheoremWF*}{Theorem WF}
\newtheorem*{TheoremBP*}{Theorem BP}

\newtheorem*{cor*} {Corollary}

 \newtheorem*{defn*} {Definition}

\newtheorem*{ZL*}{Zalcman's Lemma}
\newtheorem*{Zp*}{ZP1 Lemma}
\newtheorem*{ZP*}{ZP2 Lemma}
\newtheorem*{Ne*}{N Lemma}
\newtheorem*{claim*}{Claim}

  \theoremstyle{definition}
   \newtheorem{remark}{Remark}

\newtheorem*{examp*}{Example}
\newtheorem*{remark*}{Remark}
\newtheorem*{remarks*}{Remarks}

 \theoremstyle{remark}

\newtheorem*{notation*}{Notation}

\newcommand{\thmref}[1]{Theorem~\ref{#1}}

\newcommand{\lemref}[1]{Lemma~\ref{#1}}

 \renewcommand{\sectionmark}[1]{}

\begin{document}

\baselineskip18pt

 \title[Non explicit counterexample to a problem of quasi-normality]
 {A non explicit counterexample to a problem of quasi-normality}

 \author[Shahar Nevo]{Shahar Nevo
 }

 \address{Shahar Nevo\\ Department of Mathematics\\
 Bar-Ilan University, 52900 Ramat-Gan, Israel}
  \email{nevosh@macs.biu.ac.il}

  \author[Xuecheng Pang]{Xuecheng Pang
  }
  \address{Xuecheng Pang\\ Department of Mathematics,
East China Normal University,\linebreak Shanghai 200062, P.~R.
China} \email{xcpang@euler.math.ecnu.edu.cn}


 \begin{abstract}
In 1986, S.Y. Li and H. Xie proved the following theorem:
\textit{Let $k\ge 2$ and let $\mathcal F$ be a family of functions
meromorphic in some domain $D,$ all of whose zeros are of
multiplicity at least $k.$ Then $\mathcal F$ is normal if and only
if the family $\mathcal
F=\left\{\frac{f^{(k)}}{1+(f)^{k+1}}:f\in\mathcal F\right\}$ is
locally uniformly bounded in $D.$}

Here we give, in the case $k=2,$ a counterexample to show that if
the condition on the multiplicities of the zeros is omitted, then
the local uniform boundedness of $\mathcal F_2$ does not imply
even quasi-normality. In addition, we give a simpler proof for the
Li-Xie theorem that does not use Nevanlinna's Theory which was
used in the original proof.
    \end{abstract}



 \maketitle

\baselineskip20pt

 \section{Introduction}\label{introduction}

Marty's Theorem characterizes normality by using the first
derivative and it has an obvious geometrical meaning.

 H.L. Royden, \cite{3}, extended one direction of Marty's Theorem and proved

\begin{theorem}\label{thm1}
Let $\mathcal F$ be a family of meromorphic functions in $D,$ with
the property that for each compact set $K\subset D,$ there is a
positive increasing function $h_K$ such that
\begin{equation}\label{1}
|f'(z)|\le h_K(|f(z)|)
\end{equation} for all $f\in \mathcal F$ and $z\in K$.
Then $\mathcal F$ is normal in $D.$

 \end{theorem}

This result was extended further in various directions. In
\cite{1}, \eqref{1} is limited to only 5 values. In
\cite[Thm.2]{4}, $h_K$ is replaced by a nonnegative function that
needs to be bounded in a neighborhood of some $x_0,$ $0\le
x_0<\infty.$ Then, in \cite{6} it was shown that it is enough that
$h_K$ be finite only in a single point $x_0,$ $  x_0>0<\infty.$
Moreover, in \cite[Thm.3]{4}, this result is extended further to
higher derivatives, i.e., \eqref{1} is replaced by
 $|f^{(\ell)}(z)|\le h_K(|f(z)|)$, $f\in\mathcal F,$ $z\in K,$
where $\ell\ge 2$ and the members of $\mathcal F$ have zeros of
multiplicity $\ge l.$ The following generalization of Marty's
Theorem also deals with higher derivatives.

\begin{theorem}\label{thm2} {\rm \cite{2}}
Let $\mathcal F$ be a family of functions meromorphic on $D$ such
that each $f\in\mathcal F$ has zeros only of multiplicity $\ge k$.
Then $\mathcal F$ is normal in $D$ if and only if the family
 \begin{equation}\label{2a}\mathcal F_k=\left\{\dfrac{f^{(k)}}{1+|f^{k+1}|}:f\in\mathcal
 F\right\}\quad\text{is locally uniformly bounded in $D$.}
 \end{equation}
\end{theorem}

The direction $(\Rightarrow)$ in \thmref{thm2} is true even
without the assumption that the zeros of each $f\in\mathcal F$ are
of multiplicity at least $k$. In Section 2, we give a simpler
proof for \thmref{thm2}, without using Nevanlinna's Theory. The
condition on the multiplicities of $f\in\mathcal F$ is essential
in the direction $(\Leftarrow)$. Indeed, let $\hat{\mathcal F}_k$
be the family of all polynomials of degree at most $k-1$ in some
domain $D\subset\mathbb C.$ Then $ \frac{f^{(k)}}{1+|f|^{k+1}}=0$
for each $f\in\hat{\mathcal F}_k$, but $\hat{\mathcal F}_k$ is not
normal in $D.$ However, $\hat{\mathcal F}_k$ is a quasi-normal
family in $D$ (of order $k-1).$ The question that naturally arises
is whether the condition \eqref{2a} implies quasi-normality.

The conjecture that \eqref{2a} implies quasi-normality (without
the assumption on the multiplicities of the zeros) gets support
also from another direction.

First let us set some notation.   For $z_0\in\mathbb C$ and $r>0,$
\ $\Delta(z_0,r)=\{z:|z-z_0|<r\}. $ We write $f_n\overset
\chi\Longrightarrow f$ on $D$ to indicate that the sequence
$\{f_n\}$ converges to $f$ in the spherical metric uniformly on
compact subsets of $D$ and $f_n\Rightarrow f$ on $D$ if the
convergence is in the Euclidean metric.

Let us recall the well-known result of L. Zalcman.

\begin{lemma} [Zalcman's Lemma]\label{lem1} \cite{6a} A family
$\mathcal F$ of functions meromorphic in some domain $D$ is not
normal at  $z_0\in D$ if and only if there exist points $z_n$ in
$D,$ $z_n\to z_0;$ numbers $\rho_n\to 0^+$, and functions
$f_n\in\mathcal F$ such that
\begin{equation}\label{3a}
f_n(z_n+\rho_n\zeta)\overset\chi\Rightarrow
g(\zeta)\quad\text{in}\quad \mathbb C,\end{equation} where $g$ is
a nonconstant meromorphic function in $\mathbb C.$
\end{lemma}

 Now, suppose
that $g$ is a limit function from \eqref{3a}, and we have some
$C>0$ and $r>0$ such that
\begin{equation}\label{4a}
\frac{|f_n^{(k)}(z)|}{1+|f_n(z)|^{k+1}}\le C\quad\text{for
every}\quad z\in\Delta(z_0,r)\quad\text{and}\quad n\in\mathbb
N.\end{equation}

Let us denote the poles of $g$ (if any) by $P_g.$ Then
\begin{equation}\label{5a}
f_n(z_n+\rho_n\zeta)\Rightarrow g(\zeta)\quad\text{on}\quad
\mathbb C\setminus P_g.\end{equation} (Here we substitute
``$\overset\chi\Rightarrow$" by ``$\Rightarrow$" since in every
compact subset of $C\setminus P_g$, $f_n(z_n+\rho_n\zeta)$ is
holomorphic for large enough $n).$

Differentiating \eqref{5a} $k$ times given
$$\rho_n^k f_n^{(k)}(z_n+\rho_n\zeta)\Rightarrow
g^{(k)}(\zeta)\quad\text{in}\quad \mathbb C\setminus P_g.$$ But
then by \eqref{3a} and \eqref{4a}, we get that $g^{(k)}\equiv 0$
in $\mathbb C\setminus P_g$ and so $g^{(k)}\equiv0$ in $\mathbb
C.$ This implies that $g$ is a polynomial of degree at most $k-1.$
Hence, we get that the collection of all limit functions obtained
by \eqref{3a} is a quasi-normal family.

However, it turns out that without the condition on the
multiplicities of the zeros, the family $\mathcal F$ of
\thmref{thm2} is not quasi-normal.

We suffice to construct a detailed counterexample for the case
$k=2.$ This is the content of Section 3.

\section{Proof of \thmref{thm2}}

Assume first that $\mathcal F$ is locally uniformly bounded in
$D,$ and suppose by negation that $\mathcal F_k$ is not normal at
some $z_0\in D.$ Then similarly to \eqref{3a} we get the existence
of $f_n,z_n,\rho_n$ and $g$ such that
$f_n(z_n+\rho_n\zeta)\overset\chi\Rightarrow g(\zeta)$ in $\mathbb
C.$ With the same reasoning, we deduce that $g$ is a polynomial of
degree at most $k-1.$ But now according to the condition on the
multiplicities of the zeros of each $f_n,$ we get that the zeros
of $g$ also must be of multiplicity at least $k.$ This implies
that $g$ has no zeros and thus $g$ is a constant function, a
contradiction.

For the proof of the opposite direction, we need the following
lemma.

\begin{lemma}\label{lem2}
Let $\{f_n\}_{n=1}^\infty$ be a sequence of meromorphic functions
in a domain $D,$ satisfying $f_n\overset\chi\Rightarrow \infty$ in
$D.$ Then for every $\ell\in\mathbb N$,
$\frac{f_n^{(\ell)}}{f_n^{\ell+1}}\Rightarrow 0$ in $D.$
\end{lemma}

\begin{proof}
We apply induction. Since $\frac{1}{f_n(z)}\Rightarrow 0$ in $D,$
we can differentiate it and obtain that
$\frac{f_n'(z)}{f_n^2(z)}\Rightarrow 0$ in $D,$ and this proves
the case $\ell=1.$

\end{proof}

Assume that the lemma holds for $m\le\ell$. We prove it now for
the case $m=\ell+1.$ We have
$\frac{f_n^{(\ell)}}{f_n^{\ell+1}}(z)\Rightarrow 0$ in $D,$ and
hence, since $f_n(z)\Rightarrow\infty$ in $D,$ also
$\frac{f_n^{(\ell)}(z)}{f_n(z)^{\ell+2}}\Rightarrow 0$ in $D.$
Differentiating the last convergence gives
$$\frac{f_n^{(\ell+1)}(z)}{f_n^{\ell+2}}-(\ell+2)\frac{f_n'}{f_n^2} \frac{f_n^{(\ell)}}{f_n^{\ell+1}}(z)\Rightarrow
0\quad\text{in}\quad D.$$ The induction assumption for $m=1$ and
$m=\ell$ implies that the right term in the left hand above
converges uniformly to 0 on compacta of $D$, and thus also
$\frac{f_n^{(\ell+1)}}{f_n^{\ell+2}}(z)\Rightarrow 0$ in $D,$ as
required.

Let us prove now the opposite direction of \thmref{thm2}. Assume
that $\mathcal F$ is normal in $D$, and suppose by negation that
$\mathcal F_k$ is not locally uniformly bounded in any
neighborhood of some $z_0\in D.$ Thus, there exist functions
$f_n\in\mathcal F,$ and points $z_n\to z_0$ such that
\begin{equation}\label{6a}
\frac{f_n^{(k)} {(z_n)}}{1+|f_n^{k+1}(z_n)|}\underset
{n\to\infty}\rightarrow\infty.
\end{equation}
By the normality of $\mathcal F$, $\{f_n\}_{n=1}^\infty$ has a
subsequence that, without loss of generality, we also denote by
$\{f_n\}_{n=1}^\infty$, such that $f_n\overset\chi\Rightarrow f$
in $D.$

We separate now into cases according to the nature of $f.$

\noindent \textbf{Case 1.1} $f(z_0)\in\mathbb C.$

For small enough $r>0$, $f_n^{(k)}(z)\Rightarrow f^{(k)}(z)$ in
$\Delta(z_0,r),$ and also $1+|f_n^{k+1}(z)|\Rightarrow
1+|f(z)|^{k+1}$ in $\Delta(z_0,r).$ Since $1+|f_n(z)|^{k+1}\ge 1,$
we get that $\frac{f_n^{(k)}(z)}{1+|f_n(z)|^{k+1}}\Rightarrow
\frac{f^{(k)}(z)}{1+|f(z)|^{k+1}}$ in $\Delta(z_0,r),$ a
contradiction to \eqref{6a}.

\noindent \textbf{Case 1.2} $f(z_0)=\infty.$

Here, for small enough $r>0$, $f$ is holomorphic in
$\Delta'(z_0,r)$ and in addition $|f_n(z)|\ge 2$ and $|f(z)|\ge 2$
for large enough $n.$ Thus $\frac{f_n(z)}{1+f_n(z)^{k+1}}$ are
holomorphic in $\Delta(z_0,r)$ for large enough $n.$ We then get
by the maximum principle that
$$\frac{f_n^{(k)}(z)}{1+f_n(z)^{k+1}}\Rightarrow
\frac{f^{(k)}(z)}{1+f(z)^{k+1}}\quad\text{in}\quad \Delta(z_0,r)$$
and then for large enough $n,$
$$\max_{|z-z_0|\le r/2}\frac{|f_n^{(k)}(z)|}{1+|f_n(z)|^{k+1}}\le\max_{|z-z_0|\le r/2}\frac{|f_n^{(k)}(z)|}{|1+f_n(z)^{k+1}|}
\le\max_{|z-z_0|\le r/2}\frac{|f^{(k)}(z)|}{|1+f(z)^{k+1}|}+1.$$
The last expression is a positive constant, that does not depend
on $n$ and this is a contradiction to \eqref{6a}.

\newpage

\noindent \textbf{Case 2} $f  =\infty.$

In this case, we get by \lemref{lem2} that
$\frac{f_n^{(k)}(z)}{f_n(z)^{k+1}}\Rightarrow 0$ in $D  ,$ and
this is a contradiction to \eqref{6a}.

\section{Constructing the counterexample}

We construct a sequence of holomorphic functions
$\{f_n\}_{n=1}^\infty$, such that for every $n\ge1$ and
$z\in\Delta(0,2)$, $\frac{|f_n''(z)|}{1+|f_n(z)|^3}\le 1$ and
$\{f_n\}_{n=1}^\infty$ is not quasi-normal in $\Delta(0,2).$

Let $g_n(z)=z^n-1,$ $n\ge1.$ The zeros of $g_n$ are all simple,
$g_n(z_\ell^{(n)})=0,$ $0\le \ell\le n-1,$ where $z_\ell^{(n)}$ is
the $\ell$-th root of unity of order $n.$ Define for every
$n\ge1$, $h_n=g_ne^{p_n},$ where $p_n$ is a polynomial to be
determined. We have $h_n'=(g_n'+g_np_n')e^{p_n},$ and
$g_n'(z_\ell^{(n)})\ne0,$ $0\le\ell\le n-1.$ We want that
\begin{equation}\label{3}
p_n'(z_\ell^{(n)})=-g_n''(z_\ell^{(n)})/2g_n'(z_\ell^{(n)}),\quad
0\le\ell\le n-1\end{equation} to get that $h_n''(z_\ell^{(n)})=0.$

We have
$$h_n^{(3)}=e^{p_n}\big(g_n^{(3)}+3g_n''p_n'+\boldsymbol{3g_n'p_n''}+g_np_n^{(3)}+3g_n'p_n'{^2}+3g_np_n'p_n''+g_np_n'{}^3\big)$$

We want that
\begin{equation}\label{4}
p_n''(z_\ell^{(n)})=-(g_n^{(3)}+3g_n''p_n'+3g_n'p_n'{}^2)/3g_n'\Big|_{z=z_\ell^{(n)}},\quad
0\le\ell\le n-1
\end{equation}
to get  $h_n^{(3)}(z_\ell^{(n)})=0.$

Observe that when \eqref{3} is satisfied to determine
$p_n'(z_\ell^{(n)})$, then as in \eqref{3}, condition \eqref{4} is
in fact a condition that depends only on the values of $g_n$ and
its derivatives at the points $z_\ell^{(n)},$ $0\le\ell\le n-1.$

We have
\begin{align*}
h_n^{(4)}&=e^{p_n}\big(g_n^{(4)}+4g_n^{(3)}p_n'+6g_n''p_n''+\boldsymbol{4g_n'p_n^{(3)}}+
 g_np_n^{(4)}+6g_n''p_n'{}^2+12
 g_n'p_n'p_n''+3g_np_n''{}^2\\
 &\quad+2g_np_n'p_n^{(3)}+4g_n'p_n'{}^3+6g_np_n'{}^2p_n''+g_np_n'{}^4\big),
\end{align*}
we want that
\begin{align}
p_n^{(3)}(z_\ell^{(n)})&=-\big(g_n^{(4)}+4g_n^{(3)}p_n'+6g_n''p_n''+6g_n''p_n'{}^2+12g_n'p_n'p_n''+4g_n'p_n'{}^3\big)/4g_n'
\Big|_{z=z_\ell^{(n)}},\label{5}\\
&\quad  \quad 0\le\ell\le n-1\nonumber
\end{align}
to get $h_n^{(4)}(z_\ell^{(n)})=0.$ Observe that when \eqref{3}
and \eqref{4} are satisfied to determine $p_n'(z_\ell^{(n)})$ and
$p_n''(z_\ell^{(n)})$, then also \eqref{5} is in fact a condition
that depends only on the values of $g_n$ and its derivatives at
the points $z_\ell^{(n)}$, $0\le \ell\le n-1.$ By the theory of
interpolation \cite[p.~52]{5}, for every $n\ge1$ the conditions
\eqref{3}, \eqref{4} and \eqref{5} can be achieved with a
polynomial $p_n$ of degree at most $4n-1.$

Now, by our construction, for every $n\ge 1,$ $h_n''$ has a zero
of multiplicity at least 3 at each point $z_\ell^{(n)}, $
$0\le\ell\le n-1$, and so $\frac{h_n''}{h_n^3}$ is holomorphic (in
fact, entire) in $\Delta(0,2).$ Thus we have
$\max\limits_{z\in\overline\Delta(0,2)}|h_n''(z)/h_n^3(z)|=c_n>0.$

Define now for every $n\ge1,$ $f_n:=a_n\cdot h_n,$ where $|a_n|$
is a large enough constant such that
$\left|\frac{c_n}{a_n^2}\right|\le 1$ and such that every
subsequence of $\{f_n\}_{n=1}^\infty$ is not normal at any point
of $\partial\Delta=\{z:|z|=1\}.$ In fact, we can take $|a_n|$ to
be so large such that $f_n\to\infty$ locally uniformly in $\mathbb
C\setminus\partial\Delta.$

Now, for $z=z_\ell^{(n)},$ $0\le\ell\le n-1$,
$f_n''(z_\ell^{(n)})=0$ and thus the left hand side of \eqref{2a}
is zero. If $z\ne z_\ell^{(n)},$ $z\in\Delta(0,2),$ then
$f_n(z)\ne0$ and
$$\frac{|f_n''(z)|}{1+|f_n(z)|^3}\le
\frac{|f_n''(z)|}{|f_n(z)|^3}=\frac{1}{|a_n|^2}\,
\frac{|h_n''(z)|}{|h_n(z)|^3}\le \frac{c_n}{|a_n|^2}\le 1$$ and
\eqref{2a} is satisfied (uniformly in $\Delta(0,2)).$ This
completes the proof that $\{f_n\}_{n=1}^\infty$ has the desired
properties to be a counterexample.

\section{Some Remarks}

\begin{remark}\label{rem1}
We have not obtained an explicit formula for $f_n,$ and this
explains the title of this paper.
\end{remark}

\begin{remark}\label{rem2}
We have shown in fact a stronger counterexample: The condition
that $\left\{\frac{f''}{f^3}:f\in\mathcal F\right\}$ is locally
uniformly bounded does not imply quasi-normality of the
family~$\mathcal F.$\end{remark}

\begin{remark}\label{rem3}
An interesting open problem is to find a differential inequality
(maybe of the sort that was mentioned in this paper) that implies
quasi-normality and does not imply normality.

\end{remark}
\bibliographystyle{amsplain}

\bibliographystyle{amsplain}

\end{document}